\documentclass[a4paper,11pt]{amsart}
\usepackage{graphicx} 
\usepackage{amsfonts, amssymb, amsmath}
\usepackage{amsthm}
\usepackage{mathtools}
\usepackage[margin=1in]{geometry}
\usepackage{enumerate}
\usepackage[all]{xy}
\usepackage{color, xcolor}

\usepackage{hyperref}

\usepackage{tikz, tikz-cd}
\usetikzlibrary{3d}
\usepackage{caption}
\usepackage{subcaption}
\usepackage[all]{xy}

\newcommand{\PP}{\mathbb{P}}

\newcommand{\irr}{\operatorname{irr}}

\theoremstyle{definition}

\newtheorem{Thm}{Theorem}[section]
\newtheorem{Prop}[Thm]{Proposition}
\newtheorem{Lem}[Thm]{Lemma}
\newtheorem{Rmk}[Thm]{Remark}
\newtheorem{Cor}[Thm]{Corollary}
\newtheorem{Ex}[Thm]{Example}

\newtheorem{Question}[Thm]{Question}

\usepackage{epic}

\title{Degree of irrationality of properly elliptic surfaces}
\author{Yongnam Lee and De-Qi Zhang}
\date{August 2026}

\address{Center for Complex Geometry, Institute for Basic Science (IBS), 55 Expo-ro, Yuseong-gu, Daejeon 34126, Korea}

\email{ynlee@ibs.re.kr}

\address{Department of Mathematics, National University of Singapore, Singapore 119076, Republic of Singapore}

\email{matzdq@nus.edu.sg}


\begin{document}

\begin{abstract}
In this paper, we study the degree of irrationality of properly elliptic surfaces with a section. We prove $\min\{\chi(\mathcal O_S),\,2\operatorname{gon}(C)\}
\leq \operatorname{irr}(S)
\leq 2\operatorname{gon}(C)$.
The lower bound is obtained from the canonical bundle formula and the Cayley--Bacharach property. We show that this bound is sharp.
We also study the behavior of the degree of irrationality in moduli. A very general properly elliptic surface with a section over a curve of genus at least two has degree of irrationality at least four, 
whereas special families with $\chi(\mathcal O_S)=1$ or $2$ have degree two. Finally, we investigate properly elliptic surfaces with $\chi(\mathcal O_S)=0$, proving a generic lower bound of four and showing that
$\operatorname{irr}(C\times E)=4$ for every hyperelliptic curve $C$ of genus at least two and every elliptic curve $E$. Our paper also includes special properly elliptic surfaces without a section. 
For Dolgachev surfaces, we exclude degree two for a very general member and construct special examples of degrees two and three. 
\end{abstract}

\maketitle

\tableofcontents

\section{Introduction}

In this paper, we work over the complex numbers. Let $X$ be a smooth projective variety of dimension $n$. Its degree of irrationality is defined by
\[\operatorname{irr}(X):=
 \min\left\{
 \deg(\varphi)\ \middle|\ 
 \varphi:X\dashrightarrow \mathbb P^n
 \text{ is dominant and generically finite}
 \right\}.\]
Thus $\operatorname{irr}(X)=1$ precisely when $X$ is rational. The degree of irrationality provides a quantitative measure of how far from rationality, but it is usually difficult to compute, even for surfaces with a relatively explicit geometric structure. This topic has been studied recently  by many mathematicians. There are many interesting problems suggested in the survey paper by Chen and Martin \cite{CM26}. Many interesting results are now known for surfaces with  Kodaira dimension $\kappa=0$ and for hypersurfaces in $\mathbb P^n$, but very few are known for other cases, even elliptic surfaces with Kodaira dimension $\kappa=1$ and a section.

In this paper we study the degree of irrationality of properly elliptic surfaces. A {\it properly elliptic surface} $S$ is a smooth minimal projective surface of Kodaira dimension $\kappa(S) = 1$. Then it has an Iitaka fibration $f: S \to C$ over a smooth projective curve $C$.
The minimality of $S$ is equivalent to that $f$ being relatively minimal. Indeed, $f$ is the only elliptic fibration on $S$ (also by Kodaira's canonical bundle formula). This rigidity makes the fibration $f$ a natural tool for studying dominant rational maps from $S$ to $\mathbb P^2$.

Let $J(S)\to C$ denote the Jacobian fibration associated with $f$, and let $\alpha$ be the index of the generic fiber of \(f\). There is a dominant rational map
$S\dashrightarrow J(S)$ of degree $\alpha^{2}$ \cite{BKL76}. Since $J(S)\to C$ has a section, quotienting its generic fiber by the elliptic involution and composing with a map $C\to\mathbb P^1$ of minimal degree gives
\[
\operatorname{irr}(J(S))
 \leq 2\operatorname{gon}(C) \ \ .
\]
This upper bound is given in \cite{Yo96}. Consequently, $\operatorname{irr}(S)
 \leq 2\alpha^{2}\operatorname{gon}(C).$ We note that for every smooth projective curve $C$ of genus $g$, the {\it gonality}
$\operatorname{gon}(C)\le
\left\lfloor\frac{g+3}{2}\right\rfloor$.

Our first purpose is to complement this upper bound with a lower bound in terms of the holomorphic Euler characteristic and the gonality of the base. 
Suppose first that \(f\) has a section. In Theorem~\ref{lower bound} we prove
\[
\min\bigl\{\chi(\mathcal O_S),\,2\operatorname{gon}(C)\bigr\}
\leq\operatorname{irr}(S)
\leq 2\operatorname{gon}(C).
\]
The same lower-bound argument applies in the presence of multiple fibers. Thus, for an elliptic surface of index \(\alpha\), we obtain
$\min\bigl\{\chi(\mathcal O_S),\,2\operatorname{gon}(C)\bigr\}
\leq\operatorname{irr}(S)
\leq 2\alpha^{2}\operatorname{gon}(C)$.
This is shown in Remark~\ref{bound}.

The lower bound is obtained from the canonical bundle formula and the Cayley–Bacharach property satisfied by a general fiber of a rational map $S\dashrightarrow\mathbb P^2$. Define the fundamental line bundle by
\[\mathcal L:=(R^1f_*\mathcal O_S)^\vee .\]
Then $\deg\mathcal L=\chi(\mathcal O_S)$. 
The canonical forms used in the argument come from sections of $\omega_C\otimes\mathcal L$. When the points of a general fiber of a rational map have distinct images on \(C\), these sections separate one point from the remaining points, contradicting the null-trace Cayley–Bacharach property (cf. \cite[Proposition 4.2]{Ba12} or \cite[Section 1]{BDELU}).

An immediate consequence concerns Jacobian elliptic surfaces over an elliptic curve. If $q(S)=1$, $f$ has a section, and $\chi(\mathcal O_S)\geq4$, then $C$ is elliptic and $\operatorname{gon}(C)=2$. Hence
$\operatorname{irr}(S)=4$. We also construct, for every $k\geq1$, elliptic surfaces with a section satisfying $q(S)=1, p_g(S)=k$, and $\kappa(S)=1$. These are obtained from Weierstrass equations over an elliptic curve with fundamental line bundle of degree $k$. The lower bound in Theorem~\ref{lower bound} is sharp. For every $k\geq2$, we construct an elliptic surface with a section such that
\[
\chi(\mathcal O_S)=2k-1,\qquad
\operatorname{gon}(C)=k,\qquad
\operatorname{irr}(S)=2k-1.
\]
The construction is obtained by pulling back a rational Jacobian elliptic surface along a morphism $C\to\mathbb P^1$ of degree $2k-1$. In particular, the inequality $\chi(\mathcal O_S)\geq 2k$ in the criterion forcing 
$\operatorname{irr}(S)\geq2k$ is optimal. We also find properly elliptic surfaces $S_m$ with $\kappa(S_m) = 1, \, \chi(\mathcal O_{S_m})=m$ and $\irr(S_m)=2$ for any integer $m\geq 3$, and remark the case when $\irr(S)=\chi(\mathcal O_S)<2\operatorname{gon}(C)$. In particular, we show that
$\operatorname{irr}(S)\ge\min\{\chi(\mathcal{O}_S)+1,\,2\operatorname{gon}(C)\}$ if $h^0(C,\mathcal L)\le1$.

We next investigate what happens when $\chi(\mathcal O_S)$ is too small for the preceding lower bound to determine the degree of irrationality. Let $\mathcal F_{a,b}$ denote the moduli space of elliptic surfaces with a section satisfying
\[g(C)=a,\qquad \chi(\mathcal O_S)=b,\qquad \kappa(S)=1.\]
For $a\geq2$, we prove that a very general member of $\mathcal F_{a,b}$ satisfies
$\operatorname{irr}(S)\geq4$ in Theorem~\ref{degree4}. Degree two maps are excluded by studying involutions preserving the Iitaka fibration. Degree three maps are treated using the theory of triple planes and Pompilj’s description of irrational pencils, as recalled in \cite{CM25}. In the composed pencil case, a triple plane model forces the $j$-map of the elliptic fibration to factor through a degree three map $C\to\mathbb P^1$, a condition that does not hold for a very general Weierstrass fibration.
This generic result is complemented by special families of lower degree. For every $a\geq1$, we construct surfaces in both $\mathcal F_{a,1}$ and $\mathcal F_{a,2}$ with
$\operatorname{irr}(S)=2.$ Consequently, the degree of irrationality jumps from $2$ on special loci to at least $4$ at a very general point of these moduli spaces.

Section 4 of the paper concerns properly elliptic surfaces with
$\chi(\mathcal O_S)=0$. These surfaces have been classified in \cite{CFGLS}. If the base curve $C$ is general of genus at least three, we prove that
$\operatorname{irr}(S)\geq4$ in Theorem~\ref{chi0}. More generally, the same conclusion holds whenever
$\operatorname{gon}(C)\geq4$, without a generality assumption on $C$. The proof again excludes double and triple planes. In the triple plane case, the absence of singular fibers after the relevant base change imposes strong ramification conditions on a degree three map $C\to\mathbb P^1$. A Hurwitz space dimension count then shows that such curves form a proper sublocus of the moduli space of curves.
The genus two case exhibits different behavior. We construct elliptic surfaces with
\[
g(C)=2,\qquad
\chi(\mathcal O_S)=0,\qquad
\operatorname{irr}(S)=2
\]by making a double base change of a rational isotrivial elliptic surface with two fibers of type $I_0^{*}$. By contrast, for every smooth hyperelliptic curve $C$ of genus at least two and every elliptic curve $E$, we prove
$\operatorname{irr}(C\times E)=4$ in Theorem~\ref{product}. Chen and Martin \cite{CM23} prove $\operatorname{irr}(C_1\times C_2)=4$ when both factors are hyperelliptic of genus at least two, but their statement does not include an elliptic factor. Thus Theorem~\ref{product} gives an interesting boundary case. We also remark that $\irr(E_1 \times E_2) = 3$ where both $E_j$ are elliptic curves \cite{Lee}.

The Castelnuovo–Severi inequality helps to exclude the case when $g(C)\geq 3$. The exclusion of degree three uses Pompilj’s theorem together with the monodromy classification of rational elliptic surfaces. The only possible configurations are
$(IV,IV,IV)$ and $(IV,IV^{*})$, and their cyclic monodromy covers have genus one and zero, respectively. Neither can be dominated by a degree three cover from a curve of genus two.

We also consider Dolgachev surfaces. For a fixed coprime pair $p,q>1$, a very general Dolgachev surface $S(p,q)$ has no birational involution, and hence
$\operatorname{irr}(S(p,q))\geq3$. On the other hand, torsion sections on suitable rational Jacobian elliptic surfaces yield special examples with
$\operatorname{irr}(S(2,q))=2$, and, when $3\nmid q$, $\operatorname{irr}(S(3,q))=3.$ The latter maps are cyclic triple plane models arising from translation by a global section of order three.

\medskip

The paper is organized as follows. In Section 2 we establish the general upper and lower bounds and give examples showing their sharpness. Section 3 studies the low $\chi(\mathcal O_S)$ case, and very general Weierstrass fibrations.
Section 4 treats properly elliptic surfaces with $\chi(\mathcal O_S)=0$, including the computation $\operatorname{irr}(C\times E)=4$ for every smooth hyperelliptic curve $C$ with $g(C)\geq 2$ and $E$ is an elliptic curve. Finally, Section 5 treats Dolgachev surfaces.

\medskip

\noindent {\bf Acknowledgements.} Y. Lee would like to thank the Department of Mathematics at National University of Singapore for its hospitality during his visit, and is supported by the Institute for Basic Science (IBS-R032-D1). D. -Q. Zhang is supported by NUS ARF (A-8002487-00-00).

\section{Lower bound for the degree of irrationality}

Let $S$ be a smooth projective surface and $f:S\longrightarrow C$ a relatively minimal elliptic fibration over a smooth projective curve $C$. Write the multiple fibers as $f^*(p_i)=m_iF_i$, where $F_i$ has multiplicity one as a divisor.
Define the {\it fundamental line bundle} by $\mathcal L:=(R^1f_*\mathcal O_S)^\vee$. Then $\deg\mathcal L=\chi(\mathcal O_S)$. 

The canonical bundle formula of $S$ is
\[K_S
\simeq
f^*(\omega_C\otimes\mathcal L)
\otimes
\mathcal O_S\left(\sum_i(m_i-1)F_i\right).\]
As a \(\mathbb Q\)-divisor,
$K_S\sim_{\mathbb Q}
f^*\left(
K_C+\mathcal L+
\sum_i\left(1-\frac1{m_i}\right)p_i
\right).$

We use the following result frequently.

\begin{Lem}\label{lem:q=g}
Assume $(\chi(\mathcal{O}_S) =$) $\deg\mathcal L>0$. Then $q(S)=g(C)$, and $p_g(S)=g(C)-1+\chi(\mathcal O_S).$
\end{Lem}

\begin{proof}
By the spectral sequence and the vanishing of degree two cohomologies for coherent sheaf on $C$, we have $h^1(S, \mathcal{O}_S) = h^1(C, \mathcal{O}_C) + h^0(C, R^1f_*{\mathcal O}_S) = g(C)$, where we used  $\deg\mathcal L > 0$ for the last equality.

Note that $\sum_i(m_i-1)F_i$ is in the fixed part of $|K_S|$, so
\[p_g(S)
=h^0(C,\omega_C\otimes\mathcal L)
=g(C)-1+\deg\mathcal L = g(C)-1+\chi(\mathcal O_S).\]
\end{proof}

If $S$ has no multiple fibers then the canonical bundle formula is $K_S\simeq f^*(\omega_C\otimes\mathcal L)$, and
\[
\kappa(S)=
\begin{cases}
-\infty,&2g(C)-2+\chi(\mathcal O_S)<0,\\[2mm]
0, &2g(C)-2+\chi(\mathcal O_S)=0,\\[2mm]
1,&2g(C)-2+\chi(\mathcal O_S)>0.
\end{cases}
\]
The second row above is by the surface theory. In particular, if \(C\) is elliptic and \(f\) has a section, then $K_S=f^*\mathcal L$ and $\deg\mathcal L=\chi(\mathcal O_S)$. 
Hence, assuming further \(\chi(\mathcal O_S)>0\),
\[
\kappa(S)=1,\qquad q(S)=1,\qquad
p_g(S)=\chi(\mathcal O_S).
\]
Thus every canonical form on $S$ is pulled back from a section of $\mathcal L$.

We remark that when $f: S\to C$ is a smooth relatively minimal elliptic surface with a section
and $q(S) = 0$, we have $g(C) \le q(S) = 0$ and hence ${\rm irr}(S) \le 2 \, {\rm gon}(C) = 2$. 

In the following theorem, if $\chi(\mathcal O_S)\le 0$ then there is nothing to prove for the lower bound, so \[2 \, {\rm gon}(C) \ge \irr(S)\ge\min\left\{\chi(\mathcal O_S),\,2\operatorname{gon}(C)\right\}\] is still true. 

\begin{Thm}\label{lower bound}
Let $f:S\to C$ be a smooth projective relatively minimal elliptic surface with a section. Assume $\chi(\mathcal O_S)>0$.
Then $g(C) = q(S)$ and \[2 \, {\rm gon}(C) \ge \irr(S)\ge\min\left\{\chi(\mathcal O_S),\,2\operatorname{gon}(C)\right\}.\]
\end{Thm}

\begin{proof}
The first assertion is from Lemma \ref{lem:q=g}, while the first inequality was proved by Yoshihara \cite[Proposition 1]{Yo96}. 

Next we prove the second inequality. The canonical bundle formula gives $K_S=f^*M$ where $M: =\omega_C\otimes\mathcal L$.
If $Z\subset C$ is a zero-dimensional subscheme of length $d\le \chi(\mathcal O_S)-1$,
then 
\[H^1(C,M(-Z))
 \cong H^0(C,\mathcal L^{-1}(Z))^\vee.
\]
Since $\deg\mathcal L=\chi(\mathcal O_S)$,
\[
\deg\bigl(\mathcal L^{-1}(Z)\bigr)
=d-\deg\mathcal L\le-1.
\]
Therefore $H^1(C,M(-Z))=0$. Thus $M$ separates every subscheme of length at most $\chi(\mathcal O_S)-1$.

Suppose there is a map
\[
\phi:S\dashrightarrow\mathbb P^2
\]
which has degree $d\le \chi(\mathcal O_S)-1$. 

After resolving $\phi$, choose a general line $\ell\subset\mathbb P^2$ and let $D$ be the normalization of $\phi^{-1}(\ell)$. 
The curve $D$ is irreducible for a general $\ell$, and we have maps
\[
u:D\to\ell\simeq\mathbb P^1,\qquad \deg u=d, \qquad \text{and} \qquad v=f|_D:D\to C.
\]
Let $\Gamma$ be the normalization of the image of $(u,v):D\longrightarrow\mathbb P^1\times C$.
Write $D\xrightarrow{\,r:1\,}\Gamma\xrightarrow{\,e:1\,}\mathbb P^1$. Then $d=re$.

If $r=1$, the $d$ points $\{x_1,\ldots,x_d\}=u^{-1}(t)$ over a general $t\in\mathbb P^1$ have pairwise distinct images $b_i=f(x_i)\in C$. 
Because $d\le \chi(\mathcal O_S)-1$, the line bundle $M$ separates these \(d\) points. Hence there is
$s\in H^0(C,M)$ vanishing at $b_1,\ldots,b_{d-1}$ but not at $b_d$. Then $f^*s\in H^0(S,K_S)$ vanishes at all but one point of the general fiber of $\phi$, since $K_S=f^*M$.

This contradicts the Cayley–Bacharach property for fibers of a rational map to $\mathbb P^2$: A general fiber consists of distinct points $Z=\{x_1,\ldots,x_d\}$. The trace theorem says that $Z$ satisfies Cayley–Bacharach with respect to $|K_S|$: if a canonical form vanishes at all but one point of $Z$, it must also vanish at the remaining point. This is the null trace result used in studying degree of irrationality (cf. \cite[Proposition 4.2]{Ba12} or \cite[Section 1]{BDELU}).

Therefore, we may assume $r\ge2.$ On the other hand, $\Gamma\to C$ is surjective. Gonality cannot decrease under a surjective map of curves, so $e\ge\operatorname{gon}(C)$.
Hence $d=re\ge2\operatorname{gon}(C)$. This proves
$d\ge\min\{\chi(\mathcal O_S),2\operatorname{gon}(C)\}$.
\end{proof}

\begin{Cor}
Let $S$ be an elliptic surface with a section and $q(S)=1$. Assume $\chi(\mathcal{O}_S) \ge 4$.
Then $\irr(S)=4$.
\end{Cor}

\begin{Rmk}
For any $k \ge 1$, we can construct Kodaira dimension 1 elliptic surfaces $S_k$ with a section satisfying $q(S_k)=1$ and $p_g(S_k)=k$ by using the Weierstrass equation. This construction is well-known to experts.

Indeed, let $C$ be any smooth elliptic curve and choose a line bundle $\mathcal L\in\operatorname{Pic}^k(C)$, with $\deg\mathcal L=k$.
Choose general sections
\[
a_4\in H^0(C,\mathcal L^{\otimes4}),\qquad
a_6\in H^0(C,\mathcal L^{\otimes6}).
\]Consider the Weierstrass equation
$S_k: y^2z=x^3+a_4xz^2+a_6z^3$ inside the $\mathbb P^2$-bundle
\[
\mathbb P_C\!\left(
\mathcal O_C\oplus\mathcal L^{\otimes2}
\oplus\mathcal L^{\otimes3}
\right).
\]
Choose \(a_4,a_6\) generally enough that $\Delta=4a_4^3+27a_6^2
\in H^0(C,\mathcal L^{\otimes12})$ has only simple zeros. Then the total space $S_k$ is smooth and its singular fibers are $12k$ nodal fibers of type $I_1$. The point $[0:1:0]$ gives the zero section.
For this Weierstrass fibration,
\[
R^1f_*\mathcal O_{S_k}\simeq\mathcal L^{-1}.
\]
Because $\deg\mathcal L^{-1}=-k<0$, $H^0(C,\mathcal L^{-1})=0$. The Leray spectral sequence gives
$q(S_k)=h^1(C,\mathcal O_C)+h^0(C,\mathcal L^{-1})=1$. Moreover, $\chi(\mathcal O_{S_k})=\deg\mathcal L=k.$ Hence we obtain $p_g(S_k)=\chi(\mathcal O_{S_k})+q(S_k)-1=k$.
Finally, $mK_{S_k}=f^*(\mathcal L^{\otimes m})$ and
$h^0(S_k,\mathcal O_S(mK_{S_k}))
=h^0(C,\mathcal L^{\otimes m})
=mk.$
Thus the plurigenera grow linearly in \(m\), so $\kappa(S_k)=1$.
\end{Rmk}

\begin{Prop}\label{irr2}
Let $f:S \to C$ be a smooth projective, relatively minimal, elliptic surface with a section, and $q(S)\ge 1$.
If $\irr(S)=2$ then $\chi(\mathcal O_S)\le 2$.
\end{Prop}

\begin{proof}
Suppose the contrary that $\chi(\mathcal{O}_S) \ge 3$.
Then $g(C) = q(S) \ge 1$, and hence $\rm{gon}(C) \ge 2$. Thus $\rm{irr}(S) \ge 3$ by Theorem \ref{lower bound}, absurd.
\end{proof}



\begin{Rmk}\label{bound}
The lower bound argument in Theorem~\ref{lower bound} still works when $S\to C$ has multiple fibers. Indeed, assume $\chi(\mathcal O_S)>0$, and let
$f: S\to C$ be relatively minimal with multiple fibers $f^*(p_i)=m_iF_i$. Set again
\[
\mathcal L=(R^1f_*\mathcal O_S)^\vee,\qquad
\deg\mathcal L=\chi(\mathcal O_S),
\qquad \text{and} \qquad M=K_C\otimes\mathcal L.
\]
The canonical bundle formula is $K_S\simeq f^*M\otimes
\mathcal O_S\left(\sum_i(m_i-1)F_i\right)$. Write $V=\sum_i(m_i-1)F_i.$ There is a natural inclusion
\[
H^0(C,M)\hookrightarrow H^0(S,K_S)
\]
(with equal dimension of the domain and co-domain of this map) given by $s\longmapsto \sigma_V\, f^*s$, where $\sigma_V$ is the canonical section of $\mathcal O_S(V)$.
Away from $V$,
\[
(\sigma_Vf^*s)(x)=0
\quad\Longleftrightarrow\quad
s(f(x))=0.
\]
A general fiber of a generically finite rational map $\varphi:S\dashrightarrow\mathbb P^2$ is disjoint from the fixed divisor $V$. Therefore these canonical forms can be used in exactly the same Cayley–Bacharach argument.
Now repeat the proof of Theorem~\ref{lower bound}. If
\[
\deg\varphi=d\le\chi(\mathcal O_S)-1
\]
and the images on $C$ of the $d$ points in a general fiber are distinct, a section of $M$ separates one from the others. Multiplying its pullback by $\sigma_V$ produces a canonical form on $S$ 
vanishing at all but one point, contradicting Cayley–Bacharach.

Therefore, combined with the result in \cite{BKL76},  we have 
\[2\alpha^2\, {\rm gon}(C) \ge \irr(S)\ge\min\left\{\chi(\mathcal O_S),\,2\operatorname{gon}(C)\right\}\]
where $\alpha$ is the index of an elliptic fibration $S \to C$.
\end{Rmk}

\begin{Rmk}\label{sharp}
The bound in Theorem~\ref{lower bound} is sharp. Indeed, for every $k\ge2$, we can construct an elliptic surface $f: S \to C$ such that
\[
\chi(\mathcal O_S)=2k-1,\qquad
\operatorname{gon}(C)=k,\qquad
\operatorname{irr}(S)=2k-1<2k.
\]

Precisely, set $d=2k-1.$ Choose a general curve \(C\) of genus $g(C)=2k-3.$
Then
\[
\operatorname{gon}(C)=
\left\lfloor\frac{g(C)+3}{2}\right\rfloor=k.
\]
Choose a base point free pencil of degree $d=2k-1$: $h:C\longrightarrow\mathbb P^1.$

Let $\pi:X\to\mathbb P^1$ be a rational Jacobian elliptic surface with only nodal $I_1$-fibers. Choose $h$ so that its branch values avoid the singular fibers of $X$, and form
\[
S=X\times_{\mathbb P^1}C.
\]
Then by standard base-change properties of elliptic surfaces, $S$ is smooth and has an elliptic fibration $f: S\to C$ with a section. Its fundamental line bundle is
$\mathcal L=h^*\mathcal O_{\mathbb P^1}(1)$, so
$\deg\mathcal L=\deg h=2k-1.$ Consequently, $\chi(\mathcal O_S)=2k-1$. Moreover, $q(S)=g(C)=2k-3$ and $p_g(S)=q(S)-1+\chi(\mathcal O_S)=4k-5$.
The canonical bundle formula gives $K_S=f^*(K_C\otimes\mathcal L)$, and
\[
\deg(K_C\otimes\mathcal L)
=2(2k-3)-2+(2k-1)=6k-9>0,
\]
so $\kappa(S)=1$.
The natural projection $S\to X$ has degree $2k-1$. Since $X$ is rational, $\irr(S)\le2k-1$.
On the other hand, the lower bound proved above gives
\[
\irr(S) \ge \min\left\{\chi(\mathcal O_S), 2\operatorname{gon}(C)\right\}=
\min\{2k-1,2k\}=2k-1.
\]
Therefore $\irr(S)=2k-1$.
\end{Rmk}

\begin{Ex}
For any integer $m \ge 3$, there are properly elliptic surfaces $S_m$ with $\kappa(S_m) = 1, \, \chi(\mathcal O_{S_m})=m$ and $\irr(S_m)=2$.

Indeed, start with the rational Jacobian elliptic surface
$\pi:J\longrightarrow\mathbb P^1_{[u:v]}$. Assume $\pi$ has a smooth fiber at $0$ and $\infty$. Consider the cyclic cover
\[\nu_m:\mathbb P^1_{[U:V]}\longrightarrow\mathbb P^1_{[u:v]},
\qquad
[U:V]\longmapsto[U^m:V^m].\]
This map has degree $m$ and is totally ramified at $0$ and $\infty$, and unramified elsewhere.
Let $S_m=J\times_{\mathbb P^1,\nu_m}\mathbb P^1$
where $\pi$ induces a relatively minimal Jacobian elliptic fibration $\pi_m: S_m \to \mathbb{P}^1$ with $F$ a genearl fibre.
The fundamental line bundle is
\[\mathcal L_m=\nu_m^*\mathcal O_{\mathbb P^1}(1)
=\mathcal O_{\mathbf P^1}(m).\]
By the ramification divisor formula,
$K_{S_m} \sim \pi_m^*K_{{\mathbb P}^1} + mF \sim (m-2)F$. Therefore $\chi(\mathcal O_{S_m})=\deg\mathcal L_m=m > 0.$ Also $p_g(S)=m-1 \ge 2$ and $q(S)=0$ (cf. Lemma \ref{lem:q=g}). Thus $\kappa(S_m)=1$. Since the base curve is $\mathbb P^1$, $\irr(S_m)\le2$, while $\kappa(S_m)=1$ shows that $S_m$ is not rational. Therefore equality holds.
\end{Ex}

Before finishing this section, we remark the case when $\irr(S)=\chi(\mathcal O_S)<2\operatorname{gon}(C)$.
 
Let $\phi:S\dashrightarrow\mathbb P^2$ have degree $d$. For a general line $\ell\subset\mathbb P^2$, let $D$ be the normalization of $\phi^{-1}(\ell)$. We have
\[
u:D\longrightarrow\ell\simeq\mathbb P^1,
\qquad
v:D\longrightarrow C.
\]Assume that the $d$ points $u^{-1}(t)=\{x_1,\ldots,x_d\}$ have distinct images $b_i=v(x_i)\in C$. Put
\[
B_t=b_1+\cdots+b_d.
\]The relevant obstruction group is $H^1(C,M(-B_t))$. By Serre duality,
$H^1(C,M(-B_t))^\vee
\simeq H^0\!\left(C,\mathcal L^{-1}(B_t)\right).$ If \(d\le\chi(\mathcal{O}_S)-1\), then
\[
\deg\mathcal L^{-1}(B_t)=d-\chi(\mathcal{O}_S)<0,
\]so this group vanishes. Thus $M$ separates the points $b_1,\ldots,b_d$, contradicting Cayley–Bacharach.
But if $d=\chi(\mathcal{O}_S)$, then $\deg\mathcal L^{-1}(B_t)=0$. A degree-zero line bundle has a nonzero section precisely when it is trivial. Therefore
\[
H^1(C,M(-B_t))\ne0
\quad\text{if and only if}\quad
\mathcal O_C(B_t)\simeq\mathcal L.
\]
Thus there are two cases: If $\mathcal O_C(B_t)\not\simeq\mathcal L$, the same argument works and excludes $\deg\phi=\chi$.
The only remaining case is $B_t\in|\mathcal L|$ for general $t$.

Suppose $B_t\in|\mathcal L|.$ Then
\[
M(-B_t)
\simeq K_C\otimes\mathcal L(-B_t)
\simeq K_C.
\]
The divisors $B_t$ vary with $t$. They cannot all be equal because $v: D\to C$ is surjective. Hence
$t\longmapsto B_t$ defines a nonconstant rational curve in $|\mathcal L|$. In particular, $h^0(C,\mathcal L)\ge2$. This leads to the following proposition.

\begin{Prop}
Let \(f:S\to C\) be a relatively minimal Jacobian elliptic surface with \(\chi :=\chi(\mathcal O_S)>0\). If
$h^0(C,\mathcal L)\le1$ for $\mathcal L=(R^1f_*\mathcal O_S)^\vee$, then
\[\operatorname{irr}(S)\ge\min\{\chi+1,\,2\operatorname{gon}(C)\}.\]
\end{Prop}

\begin{proof}
Suppose
\[
d=\operatorname{irr}(S)<\min\{\chi+1,2\operatorname{gon}(C)\}.
\]
Then $d\le\chi$. If the map $D\to\Gamma$, in the notation of Theorem~\ref{lower bound}, has degree $r\ge2$, then $d\ge2\operatorname{gon}(C)$, a contradiction. Hence $r=1$, and the images of the points in a general fiber are distinct.
For $d\le\chi-1$, the original Cayley–Bacharach argument gives a contradiction. Thus only $d=\chi$ remains. Cayley–Bacharach then forces $B_t\in|\mathcal L|$ for every general $t$. Since $B_t$ varies, this implies $\dim|\mathcal L|\ge1$, contrary to $h^0(C,\mathcal L)\le1$. Therefore no such map exists.
\end{proof}

If $h^0(C,\mathcal L)\ge2$, removing the fixed part of $|\mathcal L|$ gives a pencil on $C$ of degree at most
$\deg\mathcal L=\chi$. Therefore $\operatorname{gon}(C)\le\chi$.
Consequently,
\[
\chi(\mathcal O_S)<\operatorname{gon}(C)
\quad\Longrightarrow\quad
h^0(C,\mathcal L)\le1
\quad\Longrightarrow\quad
\operatorname{irr}(S)\ge\chi(\mathcal O_S)+1.\]

The examples in Remark~\ref{sharp} satisfy
\[
\chi(\mathcal O_S)=2k-1,\qquad
\operatorname{gon}(C)=k,\qquad
\operatorname{irr}(S)=2k-1.
\]Thus $\operatorname{irr}(S)=\chi(\mathcal O_S)
<2\operatorname{gon}(C)$. Therefore the lower bound in Theorem \ref{lower bound} cannot be globally improved from $\chi$ to $\chi+1$.
In those examples, $\mathcal L=h^*\mathcal O_{\mathbb P^1}(1)$ for a degree-$2k-1$ morphism $h:C\longrightarrow\mathbb P^1$.
Hence $|\mathcal L|$ contains a pencil. The divisors $B_t$ belong to this pencil. So we may ask the following.

\medskip

\begin{Question}
Suppose $\chi := \chi(\mathcal{O}_S) <2\operatorname{gon}(C)$ and $S$ admits a degree-$\chi$ rational map $\varphi: S \dasharrow \mathbb P^2$. Then are all degree-$\chi$ rational maps essentially obtained by pulling back a rational elliptic surface along a morphism $C\to\mathbb P^1$ defined by a pencil in $|\mathcal L|$? 

It would be interesting to classify the case where $\deg \varphi = \chi$. A degree-\(\chi\) rational map would force the divisors obtained from its fibers to lie in the fixed linear system $|\mathcal L|$, and produces a rational map
$\mathbb P^2\dashrightarrow |\mathcal L|$. A particularly attractive test case is: $S\in\mathcal F_{1,3}$.
For $\mathcal F_{1,3}$, one has $|\mathcal L|\simeq\mathbb P^2$, so a degree three map would give a self-correspondence between two projective planes. This appears closely related to the exceptional Pompilj model.
\end{Question}

We end this section with the following natural question.

\begin{Question}
Does a very general $S\in\mathcal F_{1,3}$ have $\irr(S)=4$?
\end{Question}

\section{Low holomorphic Euler characteristic case}

Let $\mathcal F_{a,b}$ denote the moduli space of relatively minimal elliptic surfaces $f:S\longrightarrow C$
with a section, $g(C)=a\ge1$, and $\chi(\mathcal O_S)=b\ge1$. Then the underlying moduli space is irreducible. In the range $b>(a-1)/2$, Seiler proved that $\mathcal F_{a,b}$ is irreducible of dimension $10b+2a-2$ \cite[Theorem 7 and Lemma 13]{Sei87}. For the irreducibility statement outside this range, see Catanese \cite[Step V]{Cat07}.

\begin{Thm}\label{degree4}
For a very general member $S\in \mathcal F_{a, b}$, if $a = g(C) \ge 2$ then $\irr(S)\ge 4$.
\end{Thm} 

\begin{proof} 
We need to exclude maps of degrees two and three. A degree two rational map $S\dashrightarrow\mathbb P^2$ gives a birational involution $\iota$ of $S$. Since $\kappa(S)=1$, it preserves the Iitaka fibration and induces
$\bar\iota:C\to C$. If $\bar\iota=\mathrm{id}$, then $S/\iota$ still dominates $C$, so it cannot be rational.

If \(\bar\iota\neq\mathrm{id}\), rationality of \(S/\iota\) forces $C/\bar\iota\simeq\mathbb P^1$. Thus $C$ must be hyperelliptic and $\bar\iota$ its hyperelliptic involution.
For $a\ge3$, a general curve $C$ is not hyperelliptic. When $a=2$, an involution of $S$ covering the hyperelliptic involution $\sigma$ would require $\sigma^*L\simeq L$. Since the fixed locus of $\sigma^*$ on $\operatorname{Pic}^b(C)$ is finite, this does not occur for a general fundamental line bundle $L$. Hence a very general $S\in\mathcal F_{a,b}$ admits no degree two map to $\mathbb P^2$.


Suppose
\[
\varphi:S\dashrightarrow\mathbb P^2
\]
has degree 3. First, we treat the cyclic case. Let $\sigma$ be the deck transformation. Since $f: S\to C$ is the Iitaka fibration, $\sigma$ induces
$\beta\in\operatorname{Aut}(C), f\circ\sigma=\beta\circ f$. If $\beta=\mathrm{id}$, then $S/\langle\sigma\rangle$ still dominates $C$, contradicting its rationality. Hence $\beta$ has order $3$, and rationality gives
$C/\langle\beta\rangle\simeq\mathbf P^1$. Thus $C$ lies in the cyclic trigonal locus. A cyclic degree three cover of $\mathbf P^1$ of genus $a$ has $a+2$ totally ramified branch points, so its Hurwitz locus has dimension
$(a+2)-3=a-1$. For $a\ge2$,
\[a-1<3a-3=\dim\mathcal M_a.\]
Consequently a general curve of genus $a\ge2$ does not admit such an action.

Now, we treat non-cyclic case. Normalize $\mathbb P^2$ in $\mathbb C(S)$; this produces a triple plane birational to $S$, carrying the irrational pencil $S\to C$.
Pompilj’s theorem \cite[Theorem 1]{CM25} gives two possibilities.

(i) The exceptional case forces the base of the irrational pencil to have genus 1. Since $a\ge2$, this case is impossible:

Suppose $g:S\to E$ is a morphism to an elliptic curve. For a nonzero form $\omega_E\in H^0(E,\Omega_E^1)$, we have
$g^*\omega_E=f^*\eta$ for some $\eta\in H^0(C,\Omega_C^1)$, because $H^0(S,\Omega_S^1)=f^*H^0(C,\Omega_C^1)$. Restricting to a general fiber $F$ of $f$, $(g|_F)^*\omega_E=0.$

A non-constant map between elliptic curves pulls a nonzero one form back to a nonzero one form. Hence $g|_F$ is constant. Therefore $g$ factors through $f$: $g=h\circ f$ for some map $h:C\to E$. If $\deg h = 1$ then $f = g$ and $a = g(C) = g(E) = 1$, contradicting the assumption. Thus $\deg h>1$. The general fiber of $g$ is a disjoint union of $\deg h$ fibers of $f$. Thus $g$ does not have connected general fiber; its Stein factorization is still
\[
S\xrightarrow{f}C\xrightarrow{h}E.
\]
So any map $S\to E$ factors through $S\to C$. This is a contradiction, because in Pompilj’s theorem, an irrational pencil is required to have an irreducible general fiber.

(ii) In the composed case there is a degree three morphism $\rho:C\to\mathbb P^1$ and a rational pencil $\tau:\mathbb P^2\dashrightarrow\mathbb P^1$ such that $\tau\circ\varphi=\rho\circ f$.
For a general point $t\in\mathbb P^1$, the three elliptic fibers over $\rho^{-1}(t)=\{c_1,c_2,c_3\}$
map birationally to the same plane curve. They are therefore isomorphic, thus
\[
j_S(c_1)=j_S(c_2)=j_S(c_3).
\]
Consequently, $j_S=j_0\circ\rho$ for some $j_0:\mathbb P^1\to\mathbb P^1$.

If $a\ge5$, a general curve has no $g^1_3$. If $2\le a\le4$, write $D_A=\operatorname{div}(A)$ where $L\in\operatorname{Pic}^b(C),\ 0\ne A\in H^0(C,L^4)$. Since $\operatorname{div}_0(j_S)=3D_A$, such a factorization would imply $\rho^*Z=3D_A$ for an effective divisor $Z$ of degree $4b$ on $\mathbb P^1$. The locus of such $D_A$ has dimension at most
\[
4-a+\left\lfloor\frac{4b}{3}\right\rfloor<4b,
\]
whereas the parameter space of pairs $(L,[A])$ has dimension $4b$. Thus the factorization locus is proper. Therefore a very general member admits no degree three map to $\mathbb P^2$.
\end{proof}

By combining the upper bound in Theorem~\ref{lower bound} and the lower bound in Theorem~\ref{degree4}, we get the following.
\begin{Cor}
We have $\operatorname{irr}(S)=4$ for a very general $S\in\mathcal F_{2,b}$ with $b\geq1$.
\end{Cor}

\begin{Ex}
  Remark~\ref{sharp} shows that there is a properly elliptic surface $S$ with $\chi(\mathcal O_S)=3$, $\operatorname{gon}(C)=2$, and $\irr(S)=3$.
\end{Ex}

\begin{Ex}\label{deg2-chi2}
There are examples with $\irr(S)=2$ with $\chi(\mathcal O_S)=2$ for arbitrary $q(S)\ge 1$.

Indeed, let $\pi:X\to\mathbb P^1$ be a rational elliptic surface with a section, and let $h: C\to\mathbb P^1$ be a hyperelliptic double cover, where $C$ has arbitrary genus $g\ge1$. Choose the branch points of $h$ away from the singular fibers of $X$, and set
\[
S=X\times_{\mathbb P^1}C.
\]
Then $S$ is smooth and has an elliptic fibration $f: S\to C$ with a section. Its fundamental line bundle is
$\mathcal L=h^*\mathcal O_{\mathbb P^1}(1)$, and $\deg\mathcal L=2$.
Consequently, $\chi(\mathcal O_S)=2,$ and because $\deg\mathcal L>0$, $q(S)=g(C)=g.$ Furthermore,
\[
p_g(S)=\chi(\mathcal O_S)+q(S)-1=g+1.
\]
The canonical bundle formula gives
$K_S=f^*(K_C\otimes\mathcal L).$ Since
\[
\deg(K_C\otimes\mathcal L)
=(2g-2)+2=2g>0,
\]we have $\kappa(S)=1$. On the other hand, the natural projection $S\to X$ gives $\irr(S)=2$.
\end{Ex}

\begin{Ex}\label{deg2-chi1}
There are examples with $\irr(S)=2$, $\chi(\mathcal O_S)=1$ and arbitrary $q(S)\ge 1$.

Indeed, let $\pi:X\to\mathbb P^1$ be a rational Jacobian elliptic surface with two fibers of type $I_0^*$. Choose a hyperelliptic double cover
$h:C\to\mathbb P^1$ where $g(C)=g\ge1$, with the following branching:

- $h$ is ramified over one of the $I_0^*$-fibers,

- unramified over the other $I_0^*$-fiber, 

- all its remaining $2g+1$ branch points lie under smooth fibers.

Let \(S\to C\) be the relatively minimal model of the normalization of $X\times_{\mathbb P^1}C$.

Locally at the ramified \(I_0^*\)-fiber, the equation has the form $y^2=x^3+t^2x.$ After the ramified base change $t=u^2$,
$y^2=x^3+u^4x.$ Putting $x=u^2X, y=u^3Y$ and passing to the minimal Weierstrass model gives $Y^2=X^3+X$, which is smooth. Thus the ramified $I_0^*$-fiber disappears after passing to the minimal Weierstrass model. See also Example \ref{ex:to 4.1} for a more geometric reasoning. The other $I_0^*$-fiber is unramified, so it pulls back to two $I_0^*$-fibers. Consequently,
\[
e(S)=2e(I_0^*)=2\cdot6=12,
\]and therefore $\chi(\mathcal O_S)=\frac{e(S)}{12}=1$. Since the fundamental line bundle has positive degree, $q(S)=g(C)=g.$ Hence
\[
p_g(S)=\chi(\mathcal O_S)+q(S)-1=g.
\]
The canonical bundle formula gives
$K_S=f^*(K_C\otimes\mathcal L)$, and $\deg\mathcal L=1$. Therefore
\[
\deg(K_C\otimes\mathcal L)
=2g-2+1=2g-1>0,
\]so $\kappa(S)=1.$ Finally,  a degree two rational map $S\dashrightarrow X$ gives $\irr(S)=2$.
\end{Ex}

\section{Properly elliptic surfaces with $\chi(\mathcal O_S)=0$}

In this section, we treat the degree of irrationality of properly elliptic surfaces with $\chi(\mathcal O_S)=0$. These surfaces are classified in \cite[Lemma 1.1]{CFGLS}.
We focus on the case in which the base curve has genus greater than one.

\begin{Thm}\label{chi0}
Let $f: S\to C$ be a minimal elliptic surface with $\kappa(S)=1$ and $\chi(\mathcal O_S)=0$.
If $g(C)\ge3$ and the base curve $C$ is general in $\mathcal M_{g(C)}$, then $\irr(S)\ge4$. More generally, if $\operatorname{gon}(C)\ge4$, the same conclusion holds without any generality assumption.
\end{Thm}

\begin{proof}
Suppose $\varphi:S\dashrightarrow \mathbb P^2$ has degree 2. It determines a birational involution $\iota$ of $S$. Since $S$ is minimal with $\kappa(S)=1$, $\iota$ preserves the Iitaka fibration and induces an involution
$\bar\iota:C\longrightarrow C.$ If $\bar\iota=\mathrm{id}$, then the rational quotient $S/\iota$ still dominates $C$, impossible because a rational surface cannot dominate a curve of positive genus.
Consequently $\bar\iota\neq\mathrm{id}$, and
$C/\langle\bar\iota\rangle\simeq \mathbb P^1.$ Thus $C$ must be hyperelliptic. So $C$ cannot be a general element in $\mathcal M_{g(C)}$. Therefore, a general $C$ excludes degree 2.

We now consider the case of degree 3. Suppose $\varphi:S\dashrightarrow \mathbb P^2$ has degree 3. First we suppose the triple plane is noncyclic. Apply Pompilj’s theorem to the irrational pencil $f:S\to C$. 
The exceptional case of Pompilj's theorem forces the pencil base to have genus one, so it is impossible when $g(C)\ge2$ by the same proof of Theorem~\ref{degree4}.
Therefore the elliptic pencil must be composed with the triple plane map, giving a diagram
\[
\begin{array}{ccc}
S & \stackrel{\varphi}{\dashrightarrow} & \mathbb P^2\\
\downarrow f && \downarrow h\\
C&\xrightarrow{\ \varphi_C\ }&\mathbb P^1,
\end{array}
\qquad \deg\varphi_C=3.
\]
This is exactly the composed pencil case of Pompilj’s theorem. 
Resolve the pencil $h$ to obtain a relatively minimal rational surface with genus one fibration $r: R\longrightarrow\mathbb P^1$. We have 
\[S\sim_{\mathrm{bir}} C\times_{\mathbb P^1}R.\]
Now $\chi(\mathcal O_S)=0$ (and $K_S^2 = 0$) imply, by Noether's formula, that the Euler number $e(S) = 0$ and hence every fiber of $f$ is smooth elliptic or a multiple of a smooth elliptic curve. 
The rational elliptic surface $R$ has at least two singular fibers \cite[\S8.4]{SS}. For the complete list of possible singular-fiber configurations, see also \cite{Mir}. Over each such value, $\varphi_C$ must be totally ramified:
If the ramification type is $(2,1)$ then it leaves an unramified point above the singular value, and then the corresponding fiber of $S$ would remain singular.
This contradicts the fact that every fiber of $f$ is either a smooth elliptic curve or a multiple of one. Hence $\varphi_C: C\to\mathbb P^1$ has at least two totally ramified points.

The same conclusion is easier when $\varphi$ is cyclic. 
Indeed, let $\sigma$ be a
generator of its deck group. Since $f:S\to C$ is the Iitaka
fibration, $\sigma$ preserves this fibration. Hence there exists an automorphism $\beta\in\operatorname{Aut}(C)$ such that $f\circ\sigma=\beta\circ f.$
Since $\sigma^3=\operatorname{id}_S$, we have $\beta^3=\operatorname{id}_C$. If $\beta=\operatorname{id}_C$,
then $f$ descends to a dominant rational map $S/\langle\sigma\rangle\dashrightarrow C.$
This is impossible, since
$S/\langle\sigma\rangle\sim_{\mathrm{bir}}\mathbb P^2$ and
$g(C)>0$. Thus $\beta$ has order three. Moreover, $\varphi$ descends to map below denoted as: $\varphi_C: C \to C/\langle\beta\rangle\simeq\mathbb P^1$, and $\varphi_C$ is totally ramified and over at least two points because $\mathbb{P}^1$ with one point removed is simply connected and has no nontrivial \'etale cover.

By counting Hurwitz space dimension, $C$ cannot be a general curve in $\mathcal M_{g(C)}$:
For the degree three map $\varphi_C: C\to\mathbb P^1$, let $a$ be the number of totally ramified points and $b$ the number of simple ramification points. The Riemann–Hurwitz gives
$2a+b=2g(C)+4$. 
We have proved that $a\ge2$ in both cyclic and noncyclic case of $\varphi$, the corresponding Hurwitz locus has dimension at most
\[
a+b-3=2g(C)+1-a\le2g(C)-1.
\]
For $g(C)\ge3$,
\[
2g(C)-1<3g(C)-3=\dim\mathcal M_{g(C)}.
\]Therefore its image is a proper subvariety of $\mathcal M_g$. 
Moreover, if $\operatorname{gon}(C)\ge 4$, then none of the preceding cases can occur.
\end{proof}

\begin{Ex}\label{ex:to 4.1}
We take the Jacobian rational isotrivial elliptic surface
$R\to\mathbb P^1_{[u:v]}$ with Weierstrass equation
$y^2=x^3+a\,u^2v^2x+b\,u^3v^3$ where $4a^3+27b^2\neq0$. It has two type-$I_0^*$ fibers $F_j$ ($j = 0,1$), at $u=0$ and $v=0$.

Let $\varphi:C\longrightarrow\mathbb P^1$ be the hyperelliptic double cover of a genus two curve, chosen so that $0,\infty$ are two of its six branch values. Take the relatively minimal elliptic fibration model $S$
of the normalization of $C\times_{\mathbb P^1}R.$ The two fibres $F_j'$ of $S \to C$ lying over $F_j$ are smooth which are proper transforms of the two central components (with coefficients two) in $F_j$; indeed the double cover over $R$ is totally ramified over the $4+4$ of $(-2)$-curves (with coefficients one) in the two $F_j$ meeting the $1+1$ central components and these $4+4$ become $(-1)$-curves on the covering surface and hence blown down to points on $F_j' \subset S$. 
Thus
\[
\chi(\mathcal O_S)=0,
\qquad \kappa(S)=1,
\]and $S\to C$ is smooth with a section.
But projection to the second factor gives $S\dashrightarrow R$ of degree $2$, so $\operatorname{irr}(S)=2.$
Varying $C$, the chosen pair of Weierstrass points, and the constant elliptic curve produces a positive dimensional family dominating $\mathcal M_2$. 
Thus $g(C)=2$ cannot be included in Theorem~\ref{chi0}.
\end{Ex}

\begin{Thm}\label{product}
Let $C$ be a smooth hyperelliptic curve with $g(C)\ge 2$ and $E$ be an elliptic curve. Then $\operatorname{irr}(C\times E)=4.$
\end{Thm}

\begin{proof}
Let $S = C \times E$. Since we have the upper bound $\operatorname{irr}(S)\le4$, it is enough to exclude degree 2 and 3.

Suppose that
\[
\varphi:S \dashrightarrow\mathbb P^2
\]has degree 2. It determines a birational involution $\iota$ of $S$.
Because $S$ is minimal and $\kappa(S)=1$, $\iota$ is regular and preserves the Iitaka fibration
$p_C:S\longrightarrow C.$ Since $S/\langle\iota\rangle$ is rational,
$H^0(S,\Omega_S^1)^\iota=0$.
An involution is diagonalizable with eigenvalues \(\pm1\), so this implies
$\iota^*=-1$ on $H^0(S,\Omega_S^1)$. But
\[
H^0(S,\Omega_S^1)
=
H^0(C,K_C)\oplus H^0(E,K_E),
\]and the wedge map gives
\[
H^0(C,K_C)\otimes H^0(E,K_E)
\xrightarrow{\;\sim\;}
H^0(S,K_S).
\]Since $\iota^*=-1$ on each one form, it acts as $+1$ on their wedge. Therefore
$H^0(S,K_S)^\iota=H^0(S,K_S)\neq0.$ But the resolution of the rational quotient must have $p_g=0$, a contradiction. Thus
$\operatorname{irr}(S)\neq2.$

Now suppose that $\varphi:C\times E\dashrightarrow\mathbb P^2$ has degree 3.
Suppose $\varphi$ is a cyclic map of degree three with deck transformation $\sigma$. Since the Iitaka fibration is
$p_C:C\times E\to C$, $\sigma$ induces an automorphism $\beta$ on $C$.
If $\beta=\mathrm{id}$ then the quotient still dominates $C$, so it cannot be rational. Therefore $\beta$ has order 3, and rationality forces $C/\langle\beta\rangle\simeq\mathbb P^1.$ 
If $g(C)\geq3$, this is impossible. Indeed, $C$ also has a degree two hyperelliptic map. Since 2 and 3 are coprime, these two maps cannot factor through a common nontrivial curve. 
The Castelnuovo–Severi inequality (cf. \cite[Chapter VIII]{ACGH}) would then give $g(C)\leq(2-1)(3-1)=2$, a contradiction.
Thus $g(C)=2$. Then Riemann–Hurwitz formula shows that $C\to\mathbb P^1$ is totally ramified at four points. For such an order three action,
\[
H^0(C,K_C)=V_\zeta\oplus V_{\zeta^2},
\qquad
\dim V_\zeta=\dim V_{\zeta^2}=1,
\]where $\zeta$ is a primitive cube root.
Let $\lambda\in\{1,\zeta,\zeta^2\}$ be the linear part of the induced action on $E$.

(i) If $\lambda=1$, the quotient has an invariant holomorphic one form from $E$, so $q>0$.

(ii) If $\lambda=\zeta$, then\[
  V_{\zeta^2}\otimes H^0(E,K_E)
  \subset H^0(C\times E,K_{C\times E})
  \]is invariant.
  
(iii) If $\lambda=\zeta^2$, then $V_\zeta\otimes H^0(E,K_E)$ is invariant.

Thus in every case the quotient has either $q>0$ or $p_g>0$, so it cannot be rational. Hence no cyclic degree three map exists.

Suppose that
\[
\varphi: C\times E\dashrightarrow\mathbb P^2
\]is noncyclic of degree 3. Normalize $\mathbb C(\mathbb P^2)$ in $\mathbb C(C\times E)$; this gives a normal noncyclic triple plane birational to $C\times E$. 
Apply Pompilj’s theorem to the irrational pencil $p_C:S\longrightarrow C.$ The proof of Pompilj’s theorem gives the following precise dichotomy:

- either the given irrational pencil is composed with the triple plane map;

- or its base has genus one, and the triple plane is exceptional.

Since $g(C)\geq2$, the pencil $p_C$ must be composed with the triple plane map. Thus there is a diagram
\[
\begin{array}{ccc}
S & \stackrel{\varphi}{\dashrightarrow} & \mathbb P^2\\
\ \ \downarrow p_C && \downarrow h\\
C&\xrightarrow{\ \varphi_C\ }&\mathbb P^1,
\end{array}
\qquad
\deg\varphi_C=3.
\] where \(h\) is a rational genus-one pencil \cite[Section 2]{CM25}.

If \(g(C)\geq3\), the existence of \(\rho\) contradicts the Castelnuovo-Severi inequality, as above. Hence no noncyclic degree three map exists in this case.

Thus $g(C) = 2$. After resolving the pencil $h$, one obtains a rational surface with genus one fibration
$R\to\mathbb P^1$ such that,
\[
C\times E\sim_{\mathrm{bir}}C\times_{\mathbb P^1}R.
\]
Let $J(R)\to\mathbb P^1$ be its Jacobian. After base change by $\varphi_C$, this becomes the constant fibration $C\times E\to C$, so all its monodromy must be killed by $\varphi$.
At every singular fiber value, a degree three cover must therefore be totally ramified. Otherwise its ramification partition is $(2,1)$, and the unramified sheet remains the original nontrivial monodromy. 
Consequently every local monodromy $M\ne{\rm id}$ satisfies $M^3={\rm id}$. The only singular Kodaira fibers with monodromy of exact order 3 are $IV$ and $IV^*$ \cite[Chapter V, Sec.10, Tables 5-6]{BHPV}, with Euler numbers 4 and 8, respectively. Since a rational Jacobian elliptic surface has Euler number 12, its singular fibers must be
\[
(IV,IV,IV)
\quad\text{or}\quad
(IV,IV^*).
\]

We claim that the local monodromies belong to a single cyclic subgroup of order 3. Write the global Weierstrass equation of $J(R)$ as
\[
y^2z=x^3+A\,xz^2+B\,z^3,
\quad
A\in H^0(\mathbb P^1,\mathcal O(4)),\quad
B\in H^0(\mathbb P^1,\mathcal O(6)).
\]
At a fiber of type $IV$, one has $\operatorname{ord}A\ge2$, while at a fiber of type $IV^*$, one has $\operatorname{ord}A\ge3$. Therefore either possible configuration, $(IV,IV,IV)$ or $(IV,IV^*)$, forces $A\equiv0$. Hence $j(J(R))\equiv0$. Over the smooth locus, $J(R)$ is consequently a twist of the fixed elliptic curve $E_0:y^2=x^3+1$, so its monodromy is contained in the cyclic group $\operatorname{Aut}(E_0,0)\simeq\mu_6$. Since every local monodromy has exact order 3, the global monodromy image is the common subgroup $\mu_3\simeq\mathbb Z/3\mathbb Z$. Hence
$\operatorname{Im}\rho=\langle M\rangle\simeq\mathbb Z/3$.
By using monodromy cover, there is a smooth projective curve $C'$ such that $C'\to\mathbb P^1$ is a cyclic cover of degree three totally ramified at three or two points, respectively. 
This cover also trivializes $J(R)$. The Riemann–Hurwitz formula gives
\[g(C')=
\begin{cases}
1,&(IV,IV,IV),\\
0,&(IV,IV^*).
\end{cases}
\]
Let $C^\circ=C\setminus\phi_C^{-1}(D)$. Triviality of the pulled-back monodromy gives
\[
(\phi_C)_*\pi_1(C^\circ)\subseteq\ker\rho_{\mathrm{mon}}.
\]
Hence ${\phi_C |}_{C^\circ}$ lifts to the monodromy cover $U'\to U$, and this lift extends to a morphism $C\to C'$. Since both covers of $\mathbf P^1$ have degree 3, the map $C\to C'$ has degree $1$.
So $C$ is isomorphic to $C'$, contradicting $g(C)=2$.
Thus no noncyclic degree three map exists either. Therefore $\operatorname{irr}(C\times E)=4$.
\end{proof}

\begin{Rmk}
In Theorem~\ref{product}, the condition $g(C)\ge 2$ cannot be weakened. The degree of irrationality of a product of two elliptic curves is three \cite[Theorem 3.2]{Lee}.
\end{Rmk}

\section{Dolgachev surfaces}

A \emph{Dolgachev surface} is a properly elliptic surface with two multiple fibers of multiplicities $p$ and $q$ where $p$ and $q$ are relatively prime. Let $\pi: X\to\PP^1$ be a rational elliptic surface of index 1 and let $(p, q)$ be a pair of positive integers. We form a new algebraic surface $S(p,q)$ by performing logarithmic transforms on two of the smooth or multiplicative fibers of $\pi$, one of order $p$ and the other of order $q$. Then there is an elliptic fibration $\pi: S(p, q) \to \PP^1$ with two multiple fibers of multiplicities $p$ and $q$. Dolgachev \cite{Dol} showed that if $(p, q) = 1$, then $S(p, q)$ is algebraic and simply connected. We shall use the term \emph{Dolgachev surface of the type $(p, q)$} to mean an algebraic surface of the form $S(p, q)$ where $p$ and $q$ are relatively prime and $p,q > 1$. For given relatively prime integers p and q, the
moduli space of Dolgachev surfaces with multiple fibers of multiplicities p and q is irreducible \cite[Proposition 3.8]{FM88}.

\begin{Prop}
Let $S:=S(p, q)$ be a very general one in the moduli of Dolgachev surface of type $(p, q)$. Then $\irr(S)\ge 3$.
\end{Prop}

\begin{proof}
Suppose that
\[
\varphi:S\dashrightarrow\mathbb P^2
\]
has degree 2. Since a quadratic extension is Galois, it produces a birational involution $\iota:S\dashrightarrow S$.
Because $S$ is minimal and $\kappa(S)=1$, every birational self-map is biregular and preserves the unique Iitaka elliptic fibration $S \to \mathbb{P}^1$ (noting that the base is rational since $q(S) = 0$). Hence \(\iota\) induces an involution $\bar\iota:\mathbb P^1\longrightarrow\mathbb P^1.$

First, eliminate involutions acting on the base. Any such $\bar\iota$ must preserve the $j$-function, the two multiple-fiber points $u,v$ with their multiplicities $p,q$. Since $p\ne q$, $\bar\iota$ 
must fix $u$ and $v$ separately. Choose a general rational elliptic surface $J(S(p, q))$ with twelve $I_1$ fibers and choose $u,v$ generally. Its marked fiber configuration has trivial stabilizer in 
$\operatorname{PGL}_2$. Thus $\bar\iota=\operatorname{id}_{\mathbb P^1}$. So any involution of $S$ must act fiberwise.

Second, eliminate fiberwise involutions. For a general non-isotrivial Jacobian, $\operatorname{Aut}_K(E,0)=\{\pm1\}$. An automorphism of the torsor has a translation part and a linear part in $\{\pm1\}$.
If the linear part is $+1$ the automorphism is translation by $T\in E(K)[2]$. But a general rational elliptic surface has no nonzero rational $2$-torsion, so this case is impossible. If the linear part is -1 then  such an automorphism can exist on the torsor only if -1 stabilizes an element $\xi$ in Weil–Ch\^atelet group. So, $2\xi=0$, but $\xi$ has order $pq$ \cite[Chapter 4]{CDL}, so we get a contradiction. Thus, for a general Dolgachev surface, $S(p,q)$ has no involution, so $\operatorname{irr}(S(p,q))\ne2$.
\end{proof}

\begin{Rmk}
Since Dolgachev surface $S(p,q)$ has no irrational pencil, we cannot apply Pompilj’s theorem \cite[Theorem 1]{CM25} to study degree 3 map.
\end{Rmk}

\begin{Lem}\label{local-global}
Let $J\to B$ be a Jacobian elliptic surface, let $T$ be a torsion section of order $p$, and let $\phi:J\to J'=J/\langle T\rangle$ be the associated degree $p$ isogeny. Suppose that $S\to B$ is obtained locally at $b\in B$ 
by an $m$-fold logarithmic transform with datum $a$ of exact order $m$. Then translation by $T$ descends to $S$, and the relatively minimal model of $S/\langle T\rangle$ is locally the logarithmic transform of $J'$ with datum $\phi(a)$. In particular, its fiber multiplicity at $b$ is $\operatorname{ord}\bigl(\phi(a)\bigr)$.
\end{Lem}

\begin{proof}
By the local cyclic quotient description of a logarithmic transform, the transformed neighborhood is $\widetilde J/\langle g_a\rangle,\ g_a(x,s)=(x+a,\zeta_m s)$ \cite[Chapter V]{BHPV}. Translation by $T$ commutes with $g_a$. After taking the quotient by $\langle T\rangle$, the induced generator is
\[
g_{\phi(a)}(\bar x,s)=
(\bar x+\phi(a),\zeta_m s)
\]
on the base-changed family $J'$. If $m'=\operatorname{ord}(\phi(a))$, then $g_{\phi(a)}^{m'}$ acts only on the base coordinate. Quotienting first by this ineffective subgroup reduces the construction to the standard $m'$-fold logarithmic transform. Hence the resulting multiplicity is $m'$. The identification of the quotient generic fiber with the isogenous elliptic curve $J/\langle T\rangle$ is the quotient-by-torsion construction of Sch\"utt and Shioda, 
\cite[\S 7.7]{SS}.
\end{proof}

\begin{Ex}
Let $J\to\mathbb P^1$ be a rational Jacobian elliptic surface which admits a torsion section $T$ \cite{SS} of prime order $p=2$ or $3$. 
Choose two smooth fibers $J_u$, $J_v$ and perform logarithmic transforms of relatively prime orders $p$ and $q$, and the multiplicity-$p$ logarithmic transform datum to be generated by $T_u(a)$ where $a$ is the $p$-torsion element of the fiber at $u\in \PP^1$ given by $T$. By Lemma~\ref{local-global}, translation by $T$ extends to $S(p,q)$. Then

- the multiplicity-$p$ datum is generated by $T_u(a)$, so quotienting by $\langle T\rangle$ kills it;  let smooth surface $Y$ be the relatively minimal elliptic fibration model of the elliptic surface $S(p,q)/\langle T\rangle$;

- the multiplicity-$q$ datum retains order $q$, since $(p,q)=1$,

- $\chi(\mathcal O_Y)=1$, and

- the canonical bundle formula gives $K_Y\equiv-F_q$, hence $\kappa(Y)=-\infty$.

Since $q(Y)=0$, $Y$
is a rational genus one surface with a unique multiple fiber, of multiplicity $q$. In particular, the quotient induces a generically finite rational map
$S(p,q)\dashrightarrow Y$ of degree $p$. When $p=2$,  $\irr(S(2,q))=2$ for this special Dolgachev surface $S(2, q)$.

Take a rational elliptic surface in Tate normal form
\[
J_3:\quad
y^2z+a_1(t)xyz+a_3(t)yz^2=x^3,
\]
where $a_1\in H^0(\mathbb P^1,\mathcal O(1))$, $a_3\in H^0(\mathbb P^1,\mathcal O(3))$ are general. $T=(0:0:1)$ is a section of order 3. Its discriminant is $\Delta=a_3^3(a_1^3-27a_3)$.
Choose $q\ge2$ with $3\nmid q$, and perform logarithmic transforms of orders $3$ and $q$, using $T_u(a)$ for the order three transform. Translation by $T$ gives (birationally) a quotient map:
\[
S(3,q)\dashrightarrow Y
\] 
of degree 3, with $Y$ rational. Therefore $\irr (S(3,q))\le3$. To make the degree exactly 3, choose the coefficients, logarithmic transform points and the order $q$ torsion datum generally enough that: 

- $J(\mathbb C(\mathbb P^1))[2]=0$; 

- the marked fiber configuration has no nontrivial base involution;

- the generic $j$-invariant is neither 0 nor 1728.

Then $S(3,q)$ has no involution, so $\irr(S(3, q))=3$.
\end{Ex}

We now finish the paper by suggesting an interesting question on Dolgachev surfaces. For a Dolgachev surface $S(p,q)$, the present upper bound $\operatorname{irr}(S)\leq2p^2q^2$ is very large, while the given lower bound in Remark~\ref{bound} is essentially independent of $p,q$.

\begin{Question}
Does $\operatorname{irr}(S(p,q))\to\infty$ for a very general $S(p,q)$ as $pq\to\infty$?
\end{Question}

\end{document}